\documentclass[a4paper,10.5pt]{amsart}

\usepackage{amsmath}
\usepackage{amssymb}
\usepackage{amsthm}
\usepackage{amscd}
\usepackage[dvips]{graphicx}
\usepackage{enumerate}
\usepackage{breqn}
\usepackage{cleveref}
\usepackage{enumerate}

\newtheorem{theorem}{Theorem}[section]
\newtheorem{lemma}{Lemma}[section]

\newtheorem{prop}{Proposition}[section]

\newtheorem{form}{Formulation}
\numberwithin{equation}{section}
\newcommand{\real}{\mathbb{R}}
\newcommand{\gz}{g_{\zeta}}

\def\eqn {\begin{equation}}
\def\eeqn {\end{equation}}
\def\C{{\mathbb C}}
\def\real{{\mathbb R}}

\def\lb{\lambda}

\def\pa{\partial}

\def\A{\mathcal A}

\def\C{\mathcal C}

\def\F{\mathcal F}
\def\K{\mathcal K}
\def\L{\mathcal L}

\def\O{\mathcal O}

\def\S{\mathcal S}

\def\ka{\kappa}

\begin{document}
\title{RAPIDLY ROTATING STARS}
\author{Walter A. Strauss} 
\address{Department of Mathematics, Brown University, Providence, RI 02912}
\author{Yilun Wu}
\address{Department of Mathematics, Brown University, Providence, RI 02912}
\date{}

\begin{abstract}
A rotating star may be modeled as a continuous system of particles attracted to each other byÊ
gravity and with a given total mass and prescribed angular velocity.   
Mathematically this leads to the Euler-Poisson system. 
We prove an existence theorem for such stars that are rapidly rotating,   
depending continuously on the speed of rotation.  
This solves  a problem that has been open since Lichtenstein's work in 1933.  
The key tool is global continuation theory, combined with a delicate limiting process.  
The solutions form a connected set $\K$ in an appropriate function space. 
As the speed of rotation increases, we prove that {\it either the supports of the stars  in $\K$ become unbounded 
or the density somewhere within the stars becomes unbounded}.  
We permit any equation of state of the form $p=\rho^\gamma,\ 6/5<\gamma<2$, so long as $\gamma\ne4/3$.  
We consider two formulations, one where the angular velocity is prescribed and the other where 
the angular momentum per unit mass is prescribed.  
\end{abstract}

\maketitle

\section{Introduction}

We consider a continuum of particles attracted to each other by gravity 
but subject to no other forces.  
Initially they are static and spherical but then they begin to rotate 
around a fixed axis after some perturbation and 
thereby flatten at the poles and expand at the equator.  This is a simple 
model of a rotating star or planet.  It can also model a rotating galaxy
with its billions of stars.  In this paper we permit fast rotations and look
for steady states of the resulting configuration.  To find a family of states 
with a given mass is a highly desirable property.  We find a connected set
of such states with constant mass. 

This is a very classical problem that goes back to 
MacLaurin, Jacobi, Poincar\'e, Liapunov et al.,  who
assumed the density of the rotating fluid to be homogeneous or almost homogeneous, which is 
of course physically unrealistic. 
 See Jardetzky \cite{jardetzky2013theories} for a nice account of the classical history of the problem. 
More realistic work for slow rotations was begun by Lichtenstein \cite{lichtenstein1933untersuchungen} beginning in 1918 
and by Heilig \cite{heilig1994lichtenstein}, who
approached the problem of slowly rotating stars by means of an implicit function theorem in
function space.   They made realistic assumptions on the density 
but the mass of their solutions changes as the body changes its speed of rotation. 
Recently Jang and Makino \cite{jang2017slowly}   
studied the problem of slowly rotating stars using a simpler implicit function approach  
in the case of the power law $p=C\rho^\gamma$  and constant rotation speed.  
However, as in Lichtenstein and Heilig's work, their perturbation also does not keep the total mass constant  
and their analysis is restricted to the range $\frac65 < \gamma < \frac32$.   
In \cite{strauss2017steady} we also constructed slowly rotating stars.  
We constructed solutions with a given constant mass and permitted a general equation of state 
and a general rotation speed (see Formulation 4 in Section 7).  

A different approach was begun in 1971 by Auchmuty and Beals \cite{auchmuty1971variational} 
using a variational method with a mass constraint.  
The main difficulty in this approach is to prove that the minimizing solution has compact support.  
Their approach was generalized and extended by many authors, 
including Auchmuty \cite{auchmuty1991global}, 
Caffarelli and Friedman \cite{caffarelli1980shape}, 
Friedman and Turkington \cite{friedman1981existence}, Li \cite{li1991uniformly}, 
Chanillo and Li \cite{chanillo1994diameters}, Luo and Smoller \cite{luo2009existence}, 
Wu \cite{wu2015rotating}, and Wu \cite{wu2016existence}. 
The variational method has the major advantages
that the rotation speed is allowed to be large and that the mass is constant.   
However, there is no control on the nature of the compact support of the star, 
it does not provide a {\it continuous} curve of solutions depending on the angular velocity,  
and the equation of state is restricted to powers satisfying $\gamma>\frac43$.  
This variational method is the only one that has previously been used to prove  
the existence of solutions that rotate rapidly. 

In the present paper we extend the implicit function approach to construct solutions 
that represent stars that rotate rapidly.  
We construct, for the first time, a {\it connected} set $\K$ of solutions that is {\it global}.  
Keeping the mass constant is a key to our methodology, 
so that there is no loss or gain of particles when the star changes its rotation speed.   
Furthermore, we permit (a) the full range $\frac65<\gamma<2, \gamma\ne \frac43$,  
(b) a non-uniform angular velocity, and (c) a general equation of state $p=p(\rho)$.  

Now we describe our method.  
We begin with the steady compressible Euler-Poisson equations (EP) 
for the density $\rho\ge0$, subject to the internal forces of gravity due to the particles themselves.  
The speed $\omega(r)$ of rotation around the $x_3$-axis is allowed to depend on $r=r(x)=\sqrt{x_1^2+x_2^2}$.  
The inertial forces are entirely due to the rotation. In the region $\{x\in\real^3\ \Big|\ \rho(x)>0\}$ 
occupied by the star,  EP reduces to the equation 
\eqn \label{NG} 
\frac 1{|x|}*\rho + \kappa^2\int_0^{r} s\omega^2(s)\,ds  
     -  h(\rho)     = constant,   \eeqn
where $\omega(r)$ is a given function, $\kappa$ is a constant measuring the intensity of rotation, 
 $h$ is the enthalpy defined by $h'(\rho)=\frac{p'(\rho)}{\rho}$ with $h(0)=0$,   
and $p$ is the pressure.  The constant of gravity is assumed to be 1.  
The density must vanish at the boundary of the star.  
See the end of this introduction for the derivation of \eqref{NG}.   

So far this approach is standard.  For simplicity in this introduction let us  consider the 
standard equation of state $p(\rho)=C\rho^\gamma$.  
As a first attempt we take the inverse of $h$ to reformulate the problem as 
\eqn   \label{REF}
 \rho(\cdot)  =
\left[ \frac{1}{ |\cdot |} * \rho(\cdot) + \kappa^2 \int_0^{r} s\omega^2(s)\,ds + \alpha \right]_+^{{1}/(\gamma-1)},  
\qquad \int_{\real^3} \rho(x)\ dx = M, \eeqn  
where $\alpha$ is the negative of the constant that appears in \eqref{NG}, $M$ is the given value of the mass,  
and $[z]_+=\max(z,0)$.  
This is reminiscent of the discussion of Auchmuty \cite{auchmuty1991global} 
and the method of Jang and Makino \cite{jang2017slowly}.  
Auchmuty \cite{auchmuty1991global} found rapidly rotating solutions that unfortunately do not satisfy the physical 
boundary conditions but instead may have large density at the boundary of the star.  
What is novel in our formulation is to force the total mass $M$ to be fixed  and 
to introduce the constant $\alpha$ as a variable.   
The case $\gamma=\frac43$ is excluded because in that case  the constant mass condition introduces 
a nullspace of the linearized operator.  
If the mass were allowed to vary, the nullspace would be trivial so that the implicit function theorem 
would be applicable and $4/3$ would be permitted.  
In Section 7 we compare our approach \eqref{REF} to several alternative mathematical approaches.

Nonetheless, even with this method there is still  no way to guarantee that $\rho$ has compact support 
because the expression inside $[\dots]_+$ could be positive for large $|x|$.  
 We get the support to be compact 
by artificially forcing the parameter $\alpha$ to be sufficiently negative (see Lemma \ref{lem: support bound}).  
Then we begin the construction of rotating star solutions in the standard way by continuation 
from a non-rotating solution ($\kappa= 0$).  
It is in this first step that we require $\frac65<\gamma<2,\ \gamma\ne\frac43$,  
and we refer to \cite{strauss2017steady} for some lemmas and details.  

Letting $\kappa$ increase, we continue the construction by applying the global implicit function theorem, 
which is based on the Leray-Schauder degree (see Lemma \ref {lem: 4.5}).    
Later on, in Theorem \ref{thm: ang vel main} we obtain the whole global connected 
set $\K$ of solutions by allowing $\alpha$ to increase.  
The most novel and intricate part of our proof occurs here.  
Our main result, stated somewhat informally, is as follows. 
See Theorem \ref{thm: ang vel main} below for a completely precise version.    

\begin{theorem}  
Let $M$ be the mass of the non-rotating solution.  
Assume the pressure $p(\cdot)$ and the angular velocity $\omega(\cdot)$ satisfy 
\eqref{cond: p 1}-\eqref{cond: p 3}, \eqref{eq: def mass}-\eqref{cond: mass distinct}, 
\eqref{cond: omega 1}-\eqref{cond: omega 2}.  
By a ``solution" of the problem, we mean a triple $(\rho, \kappa, \alpha)$, where  
$\rho$ is an axisymmetric function with mass $M$ that satisfies \eqref{NG} and 
$\kappa$ refers to the intensity of rotation speed. 
Then there exists a set $\K$ of solutions satisfying the following three properties. 
\begin{itemize}
\item $\K$ is a connected set in the function space $C_c^1(\real^3)\times \real\times\real$. 
\item $\K$ contains the non-rotating solution. 
\item either 
$$\sup \{ \rho(x)\ \Big| \  x\in\real^3, (\rho,\kappa,\alpha)\in\K \}=\infty$$

or 
$$\sup\{|x|\ \Big| \ \rho(x)>0, (\rho,\kappa,\alpha)\in\K \}=\infty.$$ 
\end{itemize}
The last statement means that 
 either the densities become pointwise unbounded or the supports become unbounded.  
 \end{theorem} 
 
There is another formulation that is popular in the astronomical literature where the 
angular velocity $\omega$ is replaced by the angular momentum $L$ per unit mass.  
Our results in the latter formulation are entirely analogous, as we  describe in \Cref{sec: ang mom}.  

 We end this introduction by describing how EP  reduces to \eqref{NG}.  
 The compressible Euler-Poisson equations (EP) are 
 \begin{equation}\label{eq: full Euler-Poisson}
\begin{cases}
\rho _t + \nabla\cdot (\rho v)=0, \\
(\rho v)_t + \nabla\cdot(\rho v\otimes v) + \nabla p = \rho \nabla U, \\
U(x,t) = \int_{\real^3}\frac{\rho(x',t)}{|x-x'|}~dx'.
\end{cases}
\end{equation}
The first two equations hold where $\rho>0$, and the last equation defines $U$ on the entire $\real^3$.
To close the system, one prescribes an isentropic equation of state $p = p(\rho)$. 
To model a rotating star, one looks for a steady axisymmetric rotating solution to \eqref{eq: full Euler-Poisson}.  
That is,  we assume $\rho$ is symmetric about the $x_3$-axis  
and $v=\ka\,\omega(r)(-x_2,x_1,0) $, where $r=r(x)=\sqrt{x_1^1+x_2^2}$ with a prescribed function $\omega(r)$. 
With such specifications, the first equation in \eqref{eq: full Euler-Poisson} concerning mass conservation 
is identically satisfied. The second equation in \eqref{eq: full Euler-Poisson} 
concerning momentum conservation simplifies to
\begin{equation}\label{eq: pre Euler-Poisson}
-\rho\, \ka\,r\omega^2(r) e_r + \nabla p = \rho\,\nabla \left(\frac{1}{|\cdot|}*\rho\right), \qquad e_r=\frac{1}{r(x)}(x_1,x_2,0). 
\end{equation}
The first term in \eqref{eq: pre Euler-Poisson} can be written as
$-\rho\nabla \left(\int_0^r \omega^2(s)s~ds\right).  $
Introducing the {specific enthalpy} $h$ as above, 
\eqref{eq: pre Euler-Poisson} becomes 
\eqn \label{eq: Euler-Poisson vector}
\nabla \left(\frac{1}{|\cdot|}*\rho +\ka\int_0^r \omega^2(s)s~ds - h(\rho) \right)=0,  \eeqn
which is the same as \eqref{NG}.

\section{Properties of Non-rotating Solutions}\label{nonrot}

In this section, we summarize some properties of the non-rotating radial (spherically symmetric) solutions 
to the semilinear elliptic equation
\eqn\label{eq: EP}
\Delta u + 4\pi h^{-1}(u_+) =0 \text{ in }\real^3.
\eeqn
Such radial solutions will be the starting point of the global set of axisymmetric solutions we will construct. 

We make the following assumptions on the equation of state $p(s)$:
\eqn\label{cond: p 1}
p(s)\in C^2_{loc}(0,\infty),~p'(s)>0.
\eeqn
There exists $\gamma\in (1,2)$ such that
\eqn\label{cond: p 2}
\lim_{s\to0^+}s^{2-\gamma}p''(s)=c_0>0.
\eeqn
There exists $\gamma^*\in(\frac65, 2)$ such that
\eqn\label{cond: p 3}
\lim_{s\to\infty}s^{1-\gamma^*}p'(s)=c_1>0.
\eeqn
As shown in Lemma 3.1 of \cite{strauss2017steady}, these assumptions imply that the enthalpy $h$, 
defined by $h'(\rho)={p'(\rho)}/{\rho},\ h(0)=0$, is a one-to-one map from $[0,\infty)$ to $[0,\infty)$.  
Its inverse $h^{-1}$ is locally $C^{1,\beta}$ on $[0,\infty)$, with 
$h^{-1}(0)= (h^{-1})'(0) = 0$  and 
\eqn
\lim_{s\to\infty}\frac{h^{-1}(s)}{s}=\infty,\quad \lim_{s\to\infty}\frac{h^{-1}(s)}{s^5} = 0.
\eeqn

It follows that for all $R_0>0$, equation \eqref{eq: EP} has a positive radial (spherically symmetric) solution 
$u_0\in C^2(\overline{B_{R_0}})$ with zero boundary values on  $\pa B_{R_0} = \{x: |x|=R_0\}$   
 (see Lemma 3.2 in \cite{strauss2017steady}).  
Thus $\rho_0 := h^{-1}(u_0)$ belongs to  $C^{1,\beta}(\real^3)$ when extended to be zero outside $B_{R_0}$ 
(see Lemma 3.3 in \cite{strauss2017steady}).  
Radial solutions of \eqref{eq: EP} solve the ODE 
\eqn\label{eq: LE ode}
u'' + \frac2{|x|}u' + 4\pi h^{-1}(u_+)=0,  \eeqn 
where $'$ denotes the radial derivative.  
We denote by $u(|x|;a)$ the solution of \eqref{eq: LE ode} satisfying  $u(0;a)=a$, $u'(0;a)=0$.   
(In \cite{strauss2017steady}, $u(|x|;a)$ is denoted by $v(r;a)$.) 
For $a>0$, there are only two possibilities for the behavior of $u(|x|;a)$:
\begin{enumerate}[(i)]
\item There exists a unique $R(a)>0$ such that $u(R(a);a)=0$.
\item $u(|x|;a)>0$ for all $|x|\ge 0$.
\end{enumerate}
Let us denote by $\A$ the set of all $a$'s such that possibility (i) holds.  Note that $u_0(0)\in \A$.  
Furthermore, $\A$  is an open set, as is easily seen by considering the fact that for $a_0\in \A$ we have 
$u(R(a_0);a_0)=0$ and $u'(R(a_0);a_0)\ne0$.  
The implicit function theorem implies that $u(R(a);a)=0$ has a solution $R(a)$ for all $a$ sufficiently near $a_0$.  

 Now for $a\in \A$, we can define the \emph{physical mass} of the compactly supported radial solution $[u(|x|;a)]_+$ as
\eqn\label{eq: def mass}
M(a) = \int_{B_{R(a)}}h^{-1}(u(|x|;a))~dx = \int_0^{R(a)}4\pi h^{-1}(u(r;a))r^2~dr.
\eeqn 
Note that $M(a)>0$ and $M(\cdot)$ is differentiable on $(0,\infty)$.  
Throughout this paper we make the following assumptions on the function $M(a)$:
\eqn\label{cond: mass nonzero roc}
M'(u_0(0))\ne 0  \eeqn
and 
\eqn\label {cond: mass distinct}
M(a)\ne M(u_0(0))\ ,  \qquad \forall a\in \A, \ a\ne u_0(0).     \eeqn
Assumptions \eqref{cond: mass nonzero roc} and \eqref{cond: mass distinct} are used in 
Lemmas \ref{lem: 4.3}   and \ref{lem: 4.5}, respectively.   
Now we provide two examples of equations of state that satisfy 
both of these assumptions. 

\begin{lemma} 
Suppose that either one of the following conditions holds for the equation of state $p(s)$:
\begin{enumerate}[(a)]
\item $p(s)=s^\gamma$, where $\gamma\in(\frac65, 2),\gamma\ne \frac43$.
\item $p(s)$ satisfies \eqref{cond: p 1}, \eqref{cond: p 2}, \eqref{cond: p 3}, and 
\eqn\label{eq: p eq}
p'(s)<h(s)\le 2p'(s) \text{ for }s>0.
\eeqn
\end{enumerate}
Then $\A=(0,\infty)$, and \eqref{cond: mass nonzero roc} and \eqref{cond: mass distinct} are satisfied. 
\end{lemma}

\begin{proof}
First, if $p(s)=s^\gamma$, then $h^{-1}(s) = \left(\frac{\gamma-1}{\gamma}s\right)^{1/(\gamma-1)}$. 
By the scaling symmetry of \eqref{eq: LE ode} for this function  $h^{-1}$, we have
\eqn		\label{eq: u scaling}
u(|x|;a) = \frac{a}{a_0}u\left(\left( a/{a_0}\right)^{(2-\gamma)/(2\gamma-2)}|x|;a_0\right).  \eeqn
Thus $ \A = (0,\infty)$.
It follows from \eqref{eq: u scaling} and \eqref{eq: def mass} that 
\eqn
M(a) = \left(\frac a{a_0}\right)^{(3\gamma-4)/(2\gamma-2)} M(a_0)  \eeqn
for $a, a_0>0$.     
It is now obvious that both \eqref{cond: mass nonzero roc} and \eqref{cond: mass distinct} are satisfied if $\gamma\in(\frac65, 2)$, $\gamma\ne \frac43$.

Secondly, suppose (b) is satisfied.  The condition $h(s)\le 2p(s)$ in \eqref{eq: p eq} implies that 
$h(s)\le 2 sh'(s)$ by definition of $h$.
Thus with $t=h(s)$ we have
\eqn
t(h^{-1})'(t)\le 2h^{-1}(t) \text{ for }t>0.
\eeqn
Integration of this inequality yields 
\eqn
h^{-1}(t)\ge \frac{h^{-1}(1)}{t^2} \text{ for }0<t<1.
\eeqn
Thus the integral
$ \int_0^1 h^{-1}(t)t^{-4}~dt  $  diverges. 
So by Theorem 1 in \cite{makino1984existence}, 
no solution to \eqref{eq: LE ode} can stay positive for all $|x|$. 
This means that $\A=(0,\infty)$, so that the physical mass $M(a)$ is defined for all $a\in (0,\infty)$.   

Now if $u(|x|;a)$ is supported on the ball of radius $R(a)$, 
then $\tilde u(|x|) = u(R(a)|x|;a)$ is supported on $B_1$ and satisfies
$$
\tilde u'' + \frac2{|x|} \tilde u' + \widetilde{h^{-1}}(\tilde u_+)=0$$
where $\widetilde{h^{-1}}=R^2(a)h^{-1}$ satisfies the same kind of inequality as $h^{-1}$.   
Replacing $u$ by $\tilde u$ and $h^{-1}$ by $\widetilde {h^{-1}}$,  
we can therefore assume without loss of generality that $u(|x|;a)$ is supported on $B_1$.   
				Now the proof of Lemma 4.3  
in \cite{strauss2017steady} (without specializing the value of $a$) shows that $M'(a) = -u_a'(1;a)$.   
The subscript denotes the derivative with respect to $a$, 
while the prime denotes the derivative with respect to $|x|$. 
Letting $w=|x|u$ and $g(w,|x|) = 4\pi r h^{-1}(w/|x|)$, we have $u_a'(1;a)=w_a'(1;a)-w_a(1;a)$.  
Thus the conclusion of Lemma 4.9 in \cite{strauss2017steady} implies that $u_a'(1;a)<0$.   
Therefore both \eqref{cond: mass nonzero roc} and \eqref{cond: mass distinct} are satisfied. 
\end{proof}

\section{Formulation by Angular Velocity}
For simplicity of notation we assume $R_0=1$ for the solution $\rho_0$ in \Cref{nonrot} from now on.   
Let $M=\int_{B_1}\rho_0(x)~dx$ and 
\eqn  \label{jay}
j(x) = \int_0^{r(x)} s\ \omega^2(s)\ ds.  
\eeqn 
We will sometimes abuse notation and write $j(x)$ as $j(r(x))$. We assume that the rotation speed satisfies 
\eqn\label{cond: omega 1}
s\omega^2(s)\in L^1(0,\infty),\quad \omega^2(s)\text{ is not compactly supported},  \eeqn
and
\eqn\label{cond: omega 2}
\lim_{r(x)\to\infty} r(x)(\sup_x j - j(x))=0.  \eeqn 
This means that $\omega(r)$ decays to zero sufficiently fast as $r\to\infty$.  
It does not really matter because our purpose is to construct stars that have compact support, 
but it is a convenient assumption that was also made in \cite{auchmuty1971variational} for instance.

We define the operators 
$$
\F_1(\rho,\kappa, \alpha)  =  \rho(\cdot)  -   h^{-1}\left(
\left[ \frac{1}{ |\cdot |} * \rho(\cdot) + \kappa^2j(\cdot) + \alpha \right]_+\right),  $$
$$
\F_2(\rho)  =  \int_{\real^3} \rho(x)\ dx - M, $$
and the pair 
$$
\F(\rho,\kappa,\alpha) = (\F_1(\rho,\kappa, \alpha) , \F_2(\rho)).  $$
It is not hard to see that a solution to $\F(\rho,\kappa,\alpha)=0$ 
with $\rho\in C_{loc}(\real^3)\cap L^1(\real^3)$ will give rise to a solution of \eqref{eq: Euler-Poisson vector} with mass $M$. Indeed, on the set where $\rho$ is positive, one has
$$\frac1{|\cdot|}*\rho(x)+\kappa^2j(x)-h(\rho(x))+\alpha=0,$$
which is the same as \eqref{eq: Euler-Poisson vector}.
For fixed constants $s>3$, 
we define the weighted space 
$$ 
\C_s = \left\{  f:\real^3\to \real\ \Big|\ f \text{ is continuous, axisymmetric, even in }x_3, 
\text{ and } \|f\|_s <\infty\right\},  $$ 
where 
$$
\|f\|_s =: \sup_{x\in\real^3}\langle x\rangle^s|f(x)| <\infty . $$ 
We also define for $N>0$, 
\eqn  \label{Oh}
\O_N = \left\{ (\rho,\kappa,\alpha)\in \C_s\times\real^2\ \Big|\  \alpha +  \kappa^2 \sup_x j(x) < - \frac1N \right\}. \eeqn

We are looking for solutions of $\F(\rho,\kappa,\alpha)=0$.  
We will find them by a continuation argument starting from 
the non-rotating solution, which  satisfies 
$\F(\rho_0,0,\alpha_0)=0$.  
A key device in our proof is to control the supports of the stars.  We begin with the following 
simple, but important, observation.   

\begin{lemma} \label{lem: support bound}
For all $ (\rho,\kappa,\alpha)\in \O_N$, the expression 
$\left[ \frac{1}{ |\cdot |} * \rho(\cdot) + \kappa^2j(\cdot) + \alpha \right]_+$ 
is supported in the ball $\{ x\in\real^3\ : |x|\le C_0 N \|\rho\|_s \}$, where $C_0$ is an absolute constant.   
\end{lemma}
\begin{proof} 
First we note that $\left| \frac{1}{ |\cdot |} * \rho(\cdot) (x)\right|\le C_0 \|\rho\|_s \frac1{\langle x\rangle}$  
because $s>3$.    Hence for $|x|>C_0N\|\rho\|_s$,
$$
\left[ \frac{1}{ |\cdot |} * \rho(\cdot) (x)+ \kappa^2j(x) + \alpha \right] 
\le C_0 \|\rho\|_s \frac1{\langle x\rangle} - \frac1N < 0  $$
since $ (\rho,\kappa,\alpha)\in \O_N$.  Therefore its positive part vanishes for such $x$.  
\end{proof}

\section{Basic Properties}
\begin{lemma} \label{lem: 4.1}
$\F$ maps $\O_N$ into $\C_s\times\real$.  
It is $C^1$ Fr\'echet differentiable, with Fr\'echet derivative given by
\eqn\label{eq: frechet}
\frac{\pa\F}{\pa(\rho,\kappa, \alpha)} (\delta\rho,\delta\kappa,\delta\alpha)  
=  \left(\delta\rho  - \L(\delta\rho,\delta\kappa,\delta\alpha), \int_{\real^3}\delta\rho(x)~dx\right), 
\eeqn
where
\eqn\label{def: L}
\L(\delta\rho,\delta\kappa,\delta\alpha) = 
(h^{-1})'\left(
\left[ \frac{1}{ |\cdot |} * \rho(\cdot) + \kappa^2j(\cdot) + \alpha \right]_+\right)
\left(\frac{1}{ |\cdot |} * \delta\rho  +  2\kappa(\delta\kappa)j  +  \delta\alpha  \right).   
\eeqn
\end{lemma} 
\begin{proof}
$\F_2$ is very simple so we concentrate on $\F_1$. We need to show that $\F_1\in \C_s$.  
By Lemma \ref{lem: support bound}, we may focus on the ball $|x|\le C_0N\|\rho\|_s$.  
Since $h^{-1}$ is increasing, we have   
\begin{align}
&~\sup_{x\in\real^3} \langle x\rangle^s\ h^{-1}\left(\left[ \frac{1}{ |\cdot |} * \rho(\cdot) (x)+ \kappa^2j(x) 
+ \alpha \right]_+\right) \notag\\
\le  &~\sup_{|x|\le C_0N\|\rho\|_s}  \langle x\rangle^s\ h^{-1}\left(C_0 \|\rho\|_s \frac1{\langle x\rangle}\right)  \notag\\
\le  &~\langle C_0N\|\rho\|_s\rangle^s h^{-1}\left(C_0\|\rho\|_s\right).  
\end{align}
This shows that $\F_1(\rho,\kappa,\alpha) \in \C_s$. 
In order to prove the Fr\'echet differentiability, we again use Lemma \ref{lem: support bound} 
to deduce that 
$\left[ \frac{1}{ |\cdot |} * (\rho+\delta\rho)(\cdot)+ (\kappa+\delta\kappa)^2j + \alpha+\delta\alpha \right]_+$ 
is supported in some fixed ball $B_R$ for fixed $(\rho,\kappa,\alpha)\in \O_N$ 
and sufficiently small $(\delta\rho,\delta\kappa,\delta\alpha)$. 
Note that for $u\in \C_s$ supported in $B_R$, $\|u\|_{s}\le\langle R\rangle^s \|u\|_{C^0(\overline{B_R})}$.  
Now we only need to recognize the obvious fact that $u\mapsto h^{-1}(u_+)$ as a mapping 
from $C^0(\overline{B_R})$ to itself is differentiable with derivative $(h^{-1})'(u_+)$. 
\Cref{def: L} follows by the chain rule. 
The continuity of the Fr\'echet derivative follows in a similar way, 
as $u\mapsto (h^{-1})'(u_+)$ is continuous on $C^0(\overline{B_R})$.
\end{proof}

\begin{lemma} \label{lem: 4.2}
For each $ (\rho,\kappa,\alpha)\in \O_N$, $\frac{\partial \F}{\partial(\rho,\alpha)}(\rho,\kappa,\alpha)$ 
is a Fredholm operator on $\C_s\times\real$.
\end{lemma}
\begin{proof}
By \eqref{eq: frechet}, we only need to show $\L(\cdot,0,\cdot)$ is compact.  
By \Cref{lem: support bound}, $\L(\cdot,0,\cdot)$ is supported in $B_R$ 
with $R$ depending only on $(\rho,\kappa,\alpha)$.   
It is obvious that $\delta\rho\mapsto \frac1{|\cdot|}*\delta\rho(\cdot)$ 
is compact from $\C_s$ to $C^0(\overline{B_R})$. This implies the Fredholm property. 
\end{proof}
 
 \begin{lemma} \label{lem: 4.3}
Let  $(\rho_0,0,\alpha_0)$ be the non-rotating solution.  
If \eqref{cond: mass nonzero roc} is true, then   
the nullspace of the linear operator $\frac{\partial \F}{\partial(\rho,\alpha)} (\rho_0,0,\alpha_0) $ is trivial.    
Therefore this operator is an isomorphism.  
\end{lemma} 
\begin{proof}
From $\F(\rho_0,0,\alpha_0)=0$ we get
$$\rho_0- h^{-1}\left(\left[\frac1{|\cdot|}*\rho_0+\alpha_0\right]_+\right)=0.$$
    Denoting $u_0=h(\rho_0)$ as in Section \ref{nonrot},  we have
\eqn
u_0=\left[\frac1{|\cdot|}*\rho_0+\alpha_0\right]_+.
\eeqn
We also note the relation 
\eqn\label{eq: prime rel}
\rho_0' =( h^{-1})'(u_0)\cdot u_0'.
\eeqn
From $\frac{\partial \F}{\partial(\rho,\alpha)}(\rho_0,0,\alpha_0)(\delta\rho,\delta\alpha)=0$, we get
\eqn\label{eq: kernel 1}
\delta\rho-( h^{-1})'(u_0)\left(\frac1{|\cdot|}*\delta\rho+\delta\alpha\right)=0,
\eeqn
\eqn\label{eq: zero integral}
\int_{\real^3}\delta\rho(x)~dx=0.
\eeqn
Since $\rho_0$ and $u_0$ are supported on $B_1$, \eqref{eq: kernel 1} implies that $\delta\rho$ is also supported on $B_1$. Define $w=\frac{1}{|\cdot|}*\delta\rho+\delta\alpha$. By \eqref{eq: kernel 1}, $\delta\rho$ is H\"older continuous on $\real^3$. Thus $\Delta w = -4\pi\delta\rho$. By \eqref{eq: prime rel} and \eqref{eq: kernel 1}, we have 
\eqn\label{eq: Delta w} 
\Delta w = \begin{cases}-4\pi \frac{\rho_0'}{u_0'}w\quad &\text{ if }|x|\le 1,\\ 0 \quad &\text{ if }|x|>1.\end{cases}
\eeqn
Using spherical coordinates, we may regard $w$ as a function on $\mathbb{S}^2\times \real^+$. 
Multiplying \eqref{eq: Delta w} by the non-radial ($l\ge1$) spherical harmonic  $Y_{lm}$ 
and integrating over $\mathbb{S}^2$, we can write 
\eqn
\Delta w_{lm}-\frac{l(l+1)}{|x|^2}w_{lm} 
= \begin{cases}-4\pi \frac{\rho_0'}{u_0'}w_{lm}\quad &\text{ for }0<|x|\le 1,\\   
0 \quad &\text{ for }|x|>1,        \end{cases}
\eeqn
where $w_{lm} = \langle w,Y_{lm}\rangle_{\mathbb{S}^2}$.   
			The same argument as in the proof of Lemma 4.5 of   \cite{strauss2017steady}  
(where $w_{lm}$ is called $\varphi_{lm}$)  
will give us $w_{lm}=0$.   
There is a technical point in that argument which requires $\lim_{|x|\to 0^+}\frac{w_{lm}(|x|)}{u_0'(|x|)}=0$, 
or equivalently $\lim_{|x|\to 0^+}\frac{w_{lm}(|x|)}{|x|}=0$.
			In fact, this is true because 
\begin{align*}
\left|\frac{w_{lm}(|x|)}{|x|}\right|&=\left|\int_{\mathbb{S}^2}\frac{w(|x|\omega)}{|x|}\overline{Y_{lm}(\omega)}~d\omega\right|=\left|\int_{\mathbb{S}^2}\frac{w(|x|\omega)-w(0)}{|x|}\overline{Y_{lm}(\omega)}~d\omega\right|\\
&\le C\sup_{|y|\le |x|}|\nabla w(y)|
\end{align*}
The last quantity tends to $0$ as $|x|\to 0^+$, because $w\in C^1$, and $\nabla w(0)=0$ by the symmetry of $\delta\rho$.

We have now proven that $w$ must be a radial function.   
Integrating $\Delta w=-4\pi\delta\rho$ over $B_1$, using \eqref{eq: zero integral},  
and using the fact that $\delta\rho$ is supported on $B_1$, 
we get $w'(1)=0$.              Thus $w$ solves the boundary value problem
\eqn
\Delta w + 4\pi\frac{\rho_0'}{u_0'}w=0, \quad w'(1)=0 
\eeqn
on $B_1$.
By Lemma 4.3 of \cite{strauss2017steady}, $w$ vanishes in $B_1$. Thus $\delta\rho=0$ on $\real^3$. \Cref{eq: kernel 1} now implies $\delta\alpha=0$.   This means that the nullspace is trivial.  
\end{proof}

\begin{lemma}\label{lem: 4.4}
The nonlinear operator $(\rho,\kappa,\alpha) \mapsto  
h^{-1}\left(\left[ \frac{1}{ |\cdot |} * \rho(\cdot) + \kappa^2j(\cdot) + \alpha \right]_+\right)$
is compact from $\O_N$ into $\C_s$.
\end{lemma}
\begin{proof}
Following Nirenberg \cite{nirenberg1974topics}, 
 a continuous map $f$ is called {\it compact} if $\overline {f(K)}$ is 
a compact set for every closed bounded set $K$. 
Now by \Cref{lem: support bound}, if $(\rho,\kappa,\alpha)$ is bounded, the support of $\left[ \frac{1}{ |\cdot |} * \rho(\cdot) + \kappa^2j(\cdot) + \alpha \right]_+$ is contained in some ball $B_R$. The map is obviously compact from $\O_N$ to $C^0(\overline{B_R})$. Using again the trivial bound $\|u\|_{\C^s}\le \langle R \rangle^s \|u\|_{C^0(\overline{B_R})}$ for $u\in \C_s$ supported in $B_R$, we obtain the  compactness of this mapping into $\C_s$.
\end{proof}



\section{Global continuation}  
We now use the following form of the Global Implicit Function Theorem.
\begin{theorem}   \label{GIFT}
Let $X$ and $Z$ be Banach spaces and let $U$ be an open subset of $X\times\real$.  
Let $F:U\to Z$ be a $C^1$ mapping in the Fr\'echet sense.  
Let $(\xi_0, \kappa_0)\in U$ such that $F(\xi_0,\kappa_0)=0$.  
Assume that the linear operator $\frac{\pa F}{\pa\xi}(\xi_0,\kappa_0)$ is bijective from $X\times\real$ to $Z$.  
Assume that the mapping $(\xi.\kappa) \to F(\xi,\kappa)-\xi$ is compact from $U$ to $X$.  
Let $\S$ be the closure in $X\times\real$ of the solution set $\{(\xi,\kappa)\ \Big |\ F(\xi,\kappa)=0\}$. 
Let $\K$ be the connected component of $\S$ to which $(\xi_0,\kappa_0)$ belongs.   
Then one of the following three alternatives is valid. 
\begin{enumerate}[(i)]
\item $\K$ is unbounded in $X\times\real$.

\item $\K\backslash \{(\xi_0,\kappa_0)\}$ is connected.

\item $\K \cap \pa U \ne \emptyset$.  
\end{enumerate}
\end{theorem}
\begin{proof}
This is a standard theorem basically due to Rabinowitz, Theorem 3.2 in \cite{rabinowitz1971some}
in the case that $U=X\times\real$ and under some extra structural assumption.  
A more general version also appears in Theorem II.6.1 of \cite{kielhofer2006bifurcation};   
its proof is easy to generalize to permit a general open set $U$.  
The case of a general open set $U$ also appears explicitly in \cite{alexander1976implicit}.  
\end{proof}

\begin{lemma} \label{lem: 4.5}
There is a connected set $\K_N$ of solutions for which 
\begin{itemize}
\item either the solutions are unbounded in $\C_s\times\real^2$ 
\item or they approach the boundary of $\O_N$.  
\end{itemize}
\end{lemma}

\begin{proof}
We apply Theorem \ref{GIFT}  with $X=Z=\C_s\times\real$, $U=\O_N$   
and $ \xi=(\rho,\alpha)$.  
The starting point is  $\kappa_0=0, \ \xi_0=(\rho_0,\alpha_0)$. 
The second alternative from that theorem is that it forms a ``loop", but   
we exclude the case of a loop as follows.

Suppose there were a loop.  This means that $\K_N\backslash (\rho_0, 0 ,\alpha_0)$ is connected.  
Since $\K_N$ is connected and the operator is even in $\kappa$, it follows that 
$\K_N\backslash (\rho_0, 0 ,\alpha_0)$ must contain a different point  with $\kappa=0$, say 
$(\rho_1,0,\alpha_1)\ne (\rho_0, 0 ,\alpha_0)$.  For this new point, $\kappa=0$ means there is no rotation.  
Defining $U_1=\frac1{|x|}*\rho_1$, we have 
$$\Delta U_1 = -4\pi\rho_1 = -4\pi h^{-1}([U_1+\alpha_1]_+) := f(U_1).$$ 
This function $f$ is $C^1$.  
Of course, $\rho_1\ge0$ so that $U_1>0$ in $\real^3$.   
So we can apply Theorem 4 in \cite{gidas1979symmetry} to deduce that $\rho_1$ is radial (spherically symmetric). Letting $u_1=U_1+\alpha_1$, we get
\eqn
u_1'' + \frac2{|x|} u_1' + 4\pi h^{-1}([u_1]_+)=0,\quad u_1'(0)=0.
\eeqn
Also by \Cref{lem: support bound}, $[u_1]_+$ is compactly supported.   
If $u_1(0)\ne u_0(0)$, then by \eqref{cond: mass distinct} we would have 
$\int_{\real^3}\rho_1(x)~dx\ne \int_{\real^3}\rho_0(x)~dx = M$.   
This would violate the equation $\F_2=0$. Thus $u_1(0)=u_0(0)$. 
By uniqueness of solutions to the initial value problem of equation \eqref{eq: LE ode}, we infer that $u_1 = u_0$. 
It follows that $\rho_1=\rho_0, \alpha_1=\alpha_0$, which is a contradiction.  So there is no loop. 
We deduce that 
either $(i)$ or $(iii)$ is valid; that is, either 
$$\sup_{\K_N} \ (\|\rho\|_s + |\kappa| + |\alpha|) =\infty$$
or
$$\inf_{\K_N} \ \left |\kappa^2 \sup_x j(x) + \alpha + \frac1N \right | = 0.$$
In other words, we have either unboundedness or approach to the boundary.  
\end{proof}
          
\begin{theorem}\label{thm: ang vel main}
Define the connected set $\K  =  \bigcup_{N=1}^\infty  \K_N$.  
Uniformly along $\K$, either $\rho$ is unbounded in $L^\infty$  
or the support of $\rho$ is unbounded.   
\end{theorem}

\begin{proof} 
Because the sets $\K_N$ are nested,   
$\K$ is also connected and one of the following statements is true:
\begin{enumerate}[(a)]
\item $\sup_{\K} \ (\|\rho\|_s + |\kappa| + |\alpha|) = \infty$.
\item $\inf_\K \ |\kappa^2 \sup_x j(x) + \alpha| = 0$.  
\end{enumerate}
In order to prove the theorem, we argue by contradiction.  
Suppose that both $\sup_\K \sup_{x\in\real^3} \rho(x) <\infty$ 
and $R_* =: \sup_\K  \sup \{x\in\real^3\ \Big| \ \rho(x)\ne0\}  <\infty$.   
We will first prove that (a) is true.  

Suppose that (a) is false.  Then (b) is true and $\sup_{\K} \ (\|\rho\|_s + |\kappa| + |\alpha|) < \infty$. 
Since $|x-y|\le |x|+R_*$ for all $y$ in the support of $\rho$,   
we have $$\left(\frac1{|\cdot|}*\rho\right) (x) = \int \frac1{|x-y|} \rho(y) dy  \ge  \frac M{|x|+R_*}.$$ 
We may now write
\eqn\label{est: lower bound 1}
\frac1{|\cdot|}*\rho(x)  + \kappa^2j(x) + \alpha \ge \frac{M}{|x|+R_*} -\kappa^2 (\sup_x j - j(x))+(\kappa^2 \sup_xj(x) + \alpha).
\eeqn
Let $\kappa_0=\sup_{\K}|\kappa|$.   
Let us consider a point $x$ in the plane $x_3=0$, whence  $|x|=r(x)$. By \eqref{cond: omega 2}, $\sup_xj-j(x) = o\left(\frac1{|x|}\right)$ as $|x|\to\infty$.
Thus by \eqref{est: lower bound 1}, 
\eqn\label{est: lower bound 2}
\frac1{|\cdot|}*\rho(x)  + \kappa^2j(x) + \alpha \ge\frac{M}{|x|+R_*} - o\left(\frac{\kappa_0^2}{|x|}\right)+(\kappa^2 \sup_xj(x) + \alpha).
\eeqn
Choosing $|x|>R_*$ sufficiently large, we can make the sum of the first two terms on the right side 
of \eqref{est: lower bound 2} positive.   
Because of (b), there exists a solution $(\rho,\kappa,\alpha)\in \K$ such that 
the right side of \eqref{est: lower bound 2} is positive.  
Due to $\F_1(\rho, \kappa, \alpha) = 0$, we have $\rho(x)>0$.  
This contradicts the assumption that the support of $\rho$ is bounded by $R_*$.  

Thus (a) must be true.  Since $\rho$ is pointwise bounded and its support is also bounded all along $\K$, 
it follows that $\rho$ is also bounded in the space $\C_s$.  
Because of (a), we know that $|\kappa| + |\alpha|$ must be unbounded.  
From the definition of $\O_N$, we know that $\alpha<0$.  
In case $\kappa$ were bounded, it would have to be the case that $\alpha \to -\infty$ along a sequence.  
Then the equation $\F_1=0$ would imply that $\rho\equiv0$, which contradicts the mass constraint.  

It follows that $\kappa_n \to \infty$  for some sequence 
$(\rho_n,\kappa_n,\alpha_n)\in \K$ with $\alpha_n<0$.   
For each $n$, let us choose any point $x_n$ such that $\rho_n(x_n)>0$.   
By \eqref{cond: omega 1}, we may choose a point $y_0$ such that $r(y_0)>R_*$ and $j(y_0)>j(R_*)$.  
Since  $\rho_n(y_0)=0$ and $\rho_n(x_n)>0$, we have 
$$
0 \ge \left[\frac{1}{ |\cdot |} * \rho_n(\cdot) + \kappa_n^2j(\cdot) + \alpha_n\right]  (y_0) 
\ge \left[ \frac{1}{ |\cdot |} * \rho_n(\cdot) + \kappa_n^2j(\cdot) + \alpha_n\right] \bigg|_{x_n}^{y_0}.   $$   
On the right side, the $\alpha_n$ cancels.  
Due to our assumption that 
the values of $\rho_n$ and the supports of $\rho_n$ are uniformly bounded, 
we deduce that 
$$ 
0 \ge \kappa_n^2 [j(r(y_0))-j(r(x_n))] - C,$$
where $C$ is a fixed constant .  
Thus $j(r(x_n)) \to j(r(y_0))$ since $\kappa_n\to\infty$. 
But $r(x_n)\le R_*<r(y_0)$ and $j$ is an increasing function of $r$, so that 
$j(r(x_n)) \le j(R_*) < j(r(y_0)).$  
This is the desired contradiction.  
\end{proof}

\section{Formulation by Angular Momentum}\label{sec: ang mom}

A different formulation of the rotating star problem that is popular in the literature (see \cite{auchmuty1971variational}) is to prescribe 
the angular momentum per unit mass $L(m)$  instead of the angular velocity $\omega(r)$. 
Under this formulation the velocity field is determined  by the function $L(m)$ 
and the density $\rho(x)$ in the following way.  One first defines the mass within a cylinder by 
\eqn
m_\rho(r) = \int_{x_1^2+x_2^2\le r^2}\rho(x)~dx.  \eeqn
Then the function $L$ is related to the angular velocity $\omega(r)$ by 
\eqn
L(m_\rho(r)) = r^4\omega^2(r).  \eeqn  
In other words, $L$ is the square of $r|v|$, the angular momentum per unit mass. In this section we will entirely eliminate consideration of $\omega(r)$, and replace it by $L(m)$.  

We make the following assumptions on the function $L(m)$: 
\eqn
L\ge0,   \quad   L\in C^{1,\delta}_{loc} ([0,\infty)),  \quad L(0)=L'(0)=0    \eeqn
for some $0<\delta<1$.  
The Euler--Poisson equations are reformulated as
\eqn\label{eq: EP setup 2}
\F(\rho,\kappa,\lb)=(\F_1(\rho,\kappa,\alpha),\F_2(\rho))=0,\eeqn
where
\eqn  \label{F1L} 
\F_1(\rho,\kappa,\lb)(x) = \rho(x) -  
h^{-1}\left(\left[\frac1{|\cdot|}*\rho(x)-\kappa^2\int_{r(x)}^\infty L(m_\rho(s))s^{-3}~ds+\lb\right]_+\right),\eeqn
and
\eqn
\F_2(\rho) = \int_{\real^3}\rho(x)~dx-M.\eeqn 
Here $\lb$ plays a similar role as $\alpha$ did in the earlier formulation but it is not the same constant.  
We define $\C_s$ as above, and define
\eqn
\O_N^*=\left\{(\rho,\kappa,\lb)\in \C_s\times \real^2~\big|~\lb<-\frac1N\right\}.\eeqn

\begin{lemma}  
The analogues of Lemmas 4.1-4.4 and 5.1 are true.  
\end{lemma}

\begin{proof}
By the same argument as in \Cref{lem: support bound}, 
there is a bound on the support  of 
$$\left[\frac1{|\cdot|}*\rho(x)-\kappa^2\int_{r(x)}^\infty L(m_\rho(s))s^{-3}~ds+\lb\right]_+ .$$  
We also obtain \Cref{lem: 4.1}, with $\L$ replaced by
\begin{align}
\L(\delta\rho,\delta\kappa,\delta\alpha)(x) = &~(h^{-1})'\left(\left[\frac1{|\cdot|}*\rho(x)  
- \kappa^2\int_{r(x)}^\infty L(m_\rho(s))s^{-3}~ds + \lb\right]_+\right)\cdot                  \notag\\
&\qquad \bigg[\frac1{|\cdot|}*\delta\rho(x) - \kappa^2\int_{r(x)}^\infty L'(m_\rho(s))m_{\delta\rho}(s)s^{-3}~ds        \notag\\
&\qquad - 2\kappa(\delta\kappa)\int_{r(x)}^\infty L(m_\rho(s))s^{-3}~ds+\delta\lb\bigg].
\end{align}
				The key to justifying the Fr\'echet derivative is the estimate  
\begin{align}
&~\left|\int_{r(x)}^\infty\left[L(m_{\rho+\delta\rho}(s))-L(m_\rho(s))
- L'(m_{\rho}(s))m_{\delta\rho}(s)\right]s^{-3}~ds\right|                      \notag\\
\le &~\int_{r(x)}^\infty  \int_0^{m_{\delta\rho}(x)}  |L'(m_\rho(s)+t) - L'(m_\rho(s))| \ dt\  s^{-3}~ds       \notag\\
\le&~\|L\|_{C^{1,\delta}([0,A])} \int_{r(x)}^\infty [m_{\delta\rho}(s)]^{1+\delta}s^{-3}~ds. \label{eq: bound using L} 
\end{align}
where $A = 2\|\rho\|_{L^1}$. 
				Using  the simple fact that 
\eqn\label{eq: cylinder mass bound}
m_{\delta\rho}(r) \le  C\|\delta \rho\|_{s} \ \min(1,r^2),   \eeqn
we see that  \eqref{eq: bound using L} is uniformly bounded on compact sets by a 
constant multiple of $\|\delta\rho\|_{s}^{1+\delta}$.

\Cref{lem: 4.2} and \Cref{lem: 4.3} only involve the $\kappa=0$ case, thus they are valid without change. 
To prove \Cref{lem: 4.4}, we must show that a subsequence of $j_n(x)=\int_{r(x)}^\infty L(m_{\rho_n}(s))s^{-3}~ds$ 
converges uniformly on  compact sets if $\{\rho_n\}$ is bounded in $\C_s$. 
In fact, using \eqref{eq: cylinder mass bound} again as above, 
we see that $j_n(x)$ is uniformly bounded on a finite ball $B_R$. 
To obtain the equicontinuity of $j_n(x)$, we estimate
\begin{align*}
&~\int_{r(x)}^{r(y)}L(m_{\rho_n}(s))s^{-3}~ds\\
\le &~\|L\|_{C^{1,\delta}([0,C\|\rho_n\|_{s}])}\int_{r(x)}^{r(y)}(m_{\rho_n}(s))^{1+\delta}s^{-3}~ds\\
\le &~C\|L\|_{C^{1,\delta}([0,C\|\rho_n\|_{s}])}\|\rho_n\|_{s}^{1+\delta}\int_{r(x)}^{r(y)}s^{2\delta-1}~ds\\
\le &~C\|L\|_{C^{1,\delta}([0,C\|\rho_n\|_{s}])}\|\rho_n\|_{s}^{1+\delta}|x-y|^{\min(2\delta,1)}.
\end{align*}
We can now prove \Cref{lem: 4.5} in a similar way as before,  
thereby deducing that there is a connected set $\K^*\subset \bigcup_{N=1}^\infty \O_N^*$ 
of solutions to \eqref{eq: EP setup 2} such that at least one of the following statements is true:
\begin{enumerate}[(a)]
\item $\sup_\K(\|\rho\|_s+|\kappa|+|\lb|)=\infty$.
\item $\sup_\K \lb=0$. 
\end{enumerate}   
\end{proof}

We are now ready to prove
\begin{theorem}
Along the connected set $\K^*$, either $\rho$ is unbounded in $L^\infty$ or the support of $\rho$ is unbounded.
\end{theorem}
\begin{proof}
Arguing by contradiction, we 
suppose that $\sup_\K\|\rho\|_{L^\infty}<\infty$ and $R_*=:\sup_\K\sup\{x\in\real^3~|~\rho(x)\ne 0\}<\infty$.

Suppose also that (a) is false. Then (b) is true and $\sup_\K(\|\rho\|_s+|\kappa|+|\alpha|)<\infty$.  
We argue as in the proof of \Cref{thm: ang vel main}.  
We pick an $x$ on the $x_3=0$ plane and such that $|x|>R_*$ is sufficiently large.  
Thereby we obtain the following estimate instead of \eqref{est: lower bound 2}:
\begin{align}
&~\frac1{|\cdot|}*\rho(x) - \kappa^2 \int_{r(x)}^\infty L(m_\rho(s))s^{-3}~ds +\lb \notag\\
\ge &~\frac{M}{|x|+R_*}-\kappa_0^2\int_{r(x)}^\infty L(M)s^{-3}~ds+\lb\notag\\
\ge &~\frac{M}{|x|+R_*}-\frac{C\kappa_0^2 L(M)}{r(x)^2}+\lb\label{eq: 5.11}
\end{align}
We have used the fact that $m_\rho(s)=M$ because $r(x)>R_*$. 
Then the sum of the first two terms in \eqref{eq: 5.11} is positive. 
We now use (b) and choose a solution along $\K^*$ 
so that $\lb$ is sufficiently close to zero to make \eqref{eq: 5.11} positive. 
Hence for this solution, and this point $x$, we have  $\rho(x)>0$, contradicting the definition of $R_*$.

Thus (a) must be true. Since we assume that $\rho$ is bounded in $L^\infty$ and $R_*<\infty$, 
it follows that $\|\rho\|_s$ is also bounded. 
Suppose $|\kappa|$ is bounded.  Then $|\lb|$ must be unbounded. 
Since $\lb<0$ for solutions in $\bigcup_{N=1}^\infty \O_N^*$, it must be true that $\lb\to-\infty$ along a sequence.  
However in this case the equation $\F_1=0$ would imply that $\rho\equiv 0$ for $\lb$ sufficiently negative, 
which contradicts the mass constraint. 

It follows that $|\kappa_n|\to\infty$ and $\lb_n\to-\infty$  
along some sequence $(\rho_n,\kappa_n,\lb_n)\in\K^*$. 
Arguing as in the proof of \Cref{thm: ang vel main},  
we choose any point $y_0$ such that $r(y_0)>R_*$, and any point $x_n$ such that $\rho_n(x_n)>0$.   
So $r(x_n) < R_*$.  It follows that
\begin{align*}
0 &\ge \left[\frac{1}{ |\cdot |} * \rho_n(\cdot) 
- \kappa_n^2\int_{r(\cdot)}^\infty L(m_{\rho_n}(s))s^{-3}~ds + \lb_n\right]  (y_0) \\
&\ge \left[ \frac{1}{ |\cdot |} * \rho_n(\cdot) 
- \kappa_n^2\int_{r(\cdot)}^\infty L(m_{\rho_n}(s))s^{-3}~ds  + \lb_n\right] \bigg|_{x_n}^{y_0} \\
&\ge \kappa_n^2 \int_{r(x_n)}^{r(y_0)}L(m_{\rho_n}(s))s^{-3}~ds-C 
\ge \kappa_n^2\int_{R_*}^{r(y_0)} L(M)s^{-3}~ds-C\\
&\ge \frac{\kappa_n^2L(M)}2\left(\frac1{R_*^2}-\frac1{r^2(y_0)}\right)-C.
\end{align*}
The desired contradiction follows because $|\kappa_n|\to\infty$.
\end{proof}

\section{Comparison between Different Angular Velocity Formulations}\label{sec: form}
The rotating star problem appears in several different formulations in the literature.  
Although these formulations are not equivalent, all of them produce rotating star solutions 
to the Euler--Poisson equations under certain circumstances.  
Here we provide a comparison of the formulations in the case of prescribed angular velocity $\omega(r)$. 
The case of prescribed angular momentum per unit mass can be discussed in a similar way.  
In our discussion the density function $\rho$ is assumed to be an axisymmetric 
function on $\real^3$, $\omega(r)$ is a continuous function on $[0,\infty)$, 
and $h(s)$ is a strictly increasing continuous function from $[0,\infty)$ onto $[0,\infty)$. 
The inverse of $h$ is denoted by $h^{-1}$.
The original Euler--Poisson equation \eqref{NG} is made precise as follows.  

\begin{form}\label{form: 1}
Let $\rho$ be a non-negative function in $C_{loc}(\real^3)\cap L^1(\real^3)$.  
It is a called a rotating star solution under \Cref{form: 1} if 
there exists a real number $\alpha$ such that the equation 
\eqn\label{eq: form 1}
\frac1{|\cdot|}*\rho(x)+\int_0^{r(x)}s\omega^2(s)~ds-h(\rho(x)) + \alpha = 0   \eeqn
is valid  on the positivity set $\{x\in\real^3~|~\rho(x)>0\}$, 
\end{form}
Note that $\frac1{|\cdot|}*\rho(x)$ is defined and continuous because $\rho\in C_{loc}(\real^3)\cap L^1(\real^3)$.  
The second formulation is basically the approach taken in this paper.

\begin{form}\label{form: 2}
Let $\rho\in C_{loc}(\real^3)\cap L^1(\real^3)$.   
It is called a rotating star solution under  \Cref{form: 2} if there exists a real number $\alpha$ such that
\eqn\label{eq: form 2}
\rho(x) = h^{-1}\left(\left[\frac1{|\cdot|}*\rho(x)+\int_0^{r(x)}s\omega^2(s)~ds+\alpha\right]_+\right)
\eeqn
for all $x\in\real^3$.
\end{form}

The third formulation is basically the one used by Auchmuty in \cite{auchmuty1991global} and is closely related to the one used by Jang and Makino in \cite{jang2017slowly}.
\begin{form}\label{form: 3}
Let $\rho\in C(\overline{B_R})$ for some ball $B_R$ of radius $R$ centered at the origin.  
Extend it to be zero outside $\overline{B_R}$.  
Then $\rho$ is called a rotating star solution under  \Cref{form: 3} if there exists a real number $\alpha$ such that
\eqref{eq: form 2} is true  
for all $x\in \overline{B_R}$.
\end{form}

The fourth formulation is used by the authors in \cite{strauss2017steady}.  The density 
is explicitly designed to be a mass-invariant perturbation of a non-rotating solution.   
An earlier precursor of this formulation was used by Lichtenstein \cite{lichtenstein1933untersuchungen} and Heilig \cite{heilig1994lichtenstein}; however, their version did not keep the mass invariant.   

\begin{form}\label{form: 4}
Let $\rho_0$ be a radial (spherically symmetric) continuous function on $\real^3$ that is positive 
in a ball $B_{R_0}$ centered at the origin, vanishes in its complement, and solves the equation
\eqn\label{eq: form 4 radial}
\frac1{|\cdot|}*\rho_0(x)-h(\rho_0(x)) + \alpha_0 = 0 
\eeqn
for some real number $\alpha_0$ and all $x\in B_{R_0}$.   
Let $\zeta:\overline{B_{R_0}}\to \real$ be an axisymmetric continuous function vanishing at the origin 
to sufficiently high order such that
\eqn
\gz(x) = x\left(1+\frac{\zeta(x)}{|x|^2}\right)
\eeqn
is a homeomorphism from $\overline{B_{R_0}}$ to $\gz(\overline{B_{R_0}})$. Define
\eqn \label{rho zeta}
\rho_\zeta(x) = \frac{\int_{B_{R_0}}\rho_0(x)~dx}{\int_{\gz(B_{R_0})}\rho_0(\gz^{-1}(x))~dx}~\rho_0(\gz^{-1}(x))
\eeqn
for $x\in \gz(B_{R_0})$ and extend it to be zero elsewhere.
The function $\zeta$ is said to give rise to a rotating star solution $\rho_\zeta$ 
if there exists a real number $\alpha$ such that 
\eqn\label{eq: form 4}
\frac1{|\cdot|}*\rho_\zeta(x) + \int_0^{r(x)}s\omega^2(s)~ds-h(\rho_\zeta(x)) + \alpha = 0 
\eeqn
for all $x\in \gz(B_{R_0})$.
\end{form}
Note that the $L^1$ norm (mass) of $\rho_\zeta$ is designed to be the same as that of $\rho_0$.  
Moreover, if one can find a $\zeta$ that gives rise to a rotating star solution, 
one not only obtains some solution, but in fact the solution $\rho_\zeta$ is created by a simple deformation 
along \emph{radial directions} from the non-rotating one $\rho_0$.    
Thus a solution under \Cref{form: 4} reveals deeper structure about its relationship to a non-rotating star.

As alluded to earlier, the above formulations are not equivalent, at least as the definitions explicitly allow.  
We begin by stating how the other formulations imply \Cref{form: 1}.
\begin{prop}
The following implications hold.
\begin{enumerate}[(a)]
\item \Cref{form: 2} implies \Cref{form: 1}.
\item \Cref{form: 3}, together with the condition $\rho(x)=0$ for all $x\in \partial B_R$, implies \Cref{form: 1}.
\item \Cref{form: 4} implies \Cref{form: 1}.
\end{enumerate}
\end{prop}
\begin{proof}
To prove (a), note that if $\rho$ is a rotating star solution under \Cref{form: 2}, 
then whenever $\rho(x)>0$, the term in the square bracket of \eqref{eq: form 2} must also be positive.   
Thus in that region one can ignore the + subscript (the positive part of the square bracket), 
so that  \eqref{eq: form 1} follows.  
Assertion (b) is proven in a similar way, once it is noticed that the additional assumption 
$\rho(x)=0$ on $\partial B_R$ guarantees that $\rho\in C_{loc}(\real^3)\cap L^1(\real^3)$.  
Assertion (c) is obvious.
\end{proof}

We now discuss the weaknesses of each formulation compared with the original \Cref{form: 1}. 
The drawback of \Cref{form: 2} is that it does not capture all the solutions to \Cref{form: 1}.  
The reason is that \Cref{form: 1} does not require equality of the two sides of \eqref{eq: form 2}  
when $\rho(x)=0$, whereas \Cref{form: 2} does.  
\Cref{form: 2} requires the expression $\mathcal U(x)$ in square brackets to be 
non-positive outside the support of $\rho$, but \Cref{form: 1} does not.   
Thus \Cref{form: 2} misses many solutions which are valid according to \Cref{form: 1}, 
especially if the term involving $\omega(s)$ grows  positively at infinity.  
In that case, a valid solution under \Cref{form: 1} may make $\mathcal U(x)$  
very big for large $|x|$, while the left side remains 0.   
In fact, in order to actually work with \Cref{form: 2}, one requires the right side of \eqref{eq: form 2} 
to have enough decay near infinity, which is virtually impossible  
if the term involving $\omega(s)$ were to grow near infinity.  

\Cref{form: 3} misses some solutions of \Cref{form: 1} in the same way that \Cref{form: 2} does, 
although it does avoid the difficulty at infinity by restricting to an artificially chosen ball $B_R$.   
However, it is in general difficult to prove that $\rho(x)$ vanishes on the boundary of $B_R$.   
If one chooses $B_R$ larger than the support of a non-rotating solution, one can show that 
sufficiently small perturbations of that non-rotating solution will remain zero on the boundary of $B_R$,   
However, as soon as one continues the solution branch to fast rotations, 
nonzero boundary values may appear, which would violate the physical vacuum boundary condition 
of a rotating star.  Nor are we aware of a general mechanism that can force the support to grow gradually 
until it hits the boundary of $B_R$. In principle, the only physical solutions one can get via this approach 
seem to be merely the very small perturbations of a nonrotating star.

\Cref{form: 4} has the advantage of enforcing an equation only where $\rho_\zeta(x)>0$.   
It is thus closer in spirit to \Cref{form: 1}. However, we are not aware of any evidence that large deviations 
from a non-rotating solution will still have the structure of radial deformation that appears in \Cref{form: 4}.   
\Cref{form: 4} is also significantly more complicated than the other formulations 
when it comes to the actual construction of the function $\zeta$ (see \cite{strauss2017steady}), 
especially with regard to the required compactness property, the analogue of Lemma \ref{lem: 4.4}.    

Like \Cref{form: 1}, \Cref{form: 4} does not require \eqref{eq: form 2} on the set where $\rho(x)=0$.   
Thus it is not clear that \Cref{form: 4} implies \Cref{form: 2} or \Cref{form: 3}.   
However, in the following special situation, a solution to \Cref{form: 4} does indeed solve \Cref{form: 3}.   
For a given $\rho_0$ in \Cref{form: 4}, choose the ball $B_R$ in \Cref{form: 3} to have a fixed radius $R>R_0$.   
{\it Suppose the solution $\rho_\zeta$ is sufficiently close to the radial solution $\rho_0$} in the sense that 
$\gz(\overline{B_{R_0}})\subset B_R$, and $\rho_\zeta$ and $\rho_0$ are sufficiently close 
to each other in $C(\overline{B_R})$.   
Furthermore, suppose that $\omega(r)$ is a smooth function with sufficiently small $C(\overline{B_R})$ norm.   
Heuristically, the above conditions describe a small perturbation of the nonrotating solution $\rho_0$.   
Finally,  assume the technical condition that $r\omega(r)$ is non-decreasing.  
From \eqref{rho zeta} we see that $\rho_\zeta(x)>0$ for $x\in \gz(B_{R_0})$, 
and $\rho_\zeta(x)=0$ for $x\in \overline{B_R}\setminus \gz(B_{R_0})$.   
To prove that $\rho_\zeta$ is also a solution to \Cref{form: 3}, it remains to prove that 
\eqn
f(x) := \frac1{|\cdot|}*\rho_\zeta(x)+\int_0^{r(x)}s\omega^2(s)~ds + \alpha\le 0
\eeqn
for $x\in \overline{B_R}\setminus \gz(B_{R_0})$.   
In this ``annular" region 
we have $\Delta f(x) = \Delta \int_0^{r(x)}s\omega^2(s)~ds = \frac1r(r^2\omega^2(r))'\ge 0$.   
Hence we only have to show $f(x)\le 0$ on $\gz(\partial B_{R_0}) \cup \partial B_R$.   
By \eqref{eq: form 4} and the continuity of $\rho_\zeta$, 
we obviously have $f(x)=h(\rho_\zeta)=0$ for $x\in \gz(\partial B_{R_0})$. 
It remains to prove that $f(x)\le 0$ on $\partial B_R$.  
For this purpose note that the function $f_0(x) := \frac1{|\cdot|}*\rho_0 (x) + \alpha_0$ is harmonic outside $B_{R_0}$   
and that $f_0=0$ and $f_0'<0$ on $\partial B_{R_0}$.  
It follows that $f_0(x)<0$ for $|x|>R_0$.  
Since $f(x)$ is sufficiently close to $f_0(x)$ in supremum norm  by the smallness assumptions,  
we have $f(x)<0$ on $\pa B_R$.    
This shows that $\rho_\zeta$ is also a solution to \Cref{form: 3}.   
Since the typical construction of solutions via \Cref{form: 3}  guarantees local uniqueness, 
this reasoning shows that the unique solution must have the structure detailed in \Cref{form: 4}.

If the rotating star problem is treated as a classical free boundary problem, then a 
fifth possible formulation emerges.   Let us begin with \Cref{form: 1} with a connected set $\Omega=\{\rho>0\}$.  
Let $q=h(\rho)$.  Taking the Laplacian of \eqref{NG}, the function $q$ satisfies the elliptic equation 
\eqn 
\Delta q = 4\pi h^{-1}(q) - \kappa^2\Delta j    \eeqn 
in $\Omega$ with $j$ defined by \eqref{jay}, together with the pair of boundary conditions 
\eqn \label{hodographBC}
q=0 \quad \text{ and } \quad \frac1{|\cdot|}*h^{-1}(q)  +  \kappa^2 j =  \text{ constant on } \pa\Omega.  \eeqn  
Now we use a transformation of hodograph type to convert $\Omega$ to a fixed domain.  
Using standard spherical coordinates $(s,\theta,\phi)$, we exchange independent and dependent variables by 
defining 
\eqn 
s' = 1 - q(s,\theta,\phi)\quad \text{ and } \quad  w(s',\theta,\phi)=s .  \eeqn
Then $\Omega$ goes into the unit ball while its boundary $\pa\Omega$ 
goes into the unit sphere $\pa B_1 = \{s'=1\}$.    
The first boundary condition in \eqref{hodographBC} is automatically satisfied.  The whole problem is thereby transformed into a nonlinear  
elliptic equation for $w(s',\theta,\phi)$ in the unit ball $B_1$ with a single nonlinear boundary condition.  
This is Formulation 5.  
We continue to assume axisymmetry, which means that $w$ does not depend on $\phi$.  
This formulation has the primary advantage that the domain is fixed.  
However it appears to be rather complicated to analyze because both the equation 
and the boundary condition are highly nonlinear and have variable coefficients.  
We refrain from providing the details. \bigskip

\noindent{\bf Acknowledgments.}  
YW is supported by NSF Grant  DMS-1714343.  
WS and YW also acknowledge the support of the spring 2017 semester program at ICERM (Brown U.), 
where much of this work was done.

\bibliographystyle{acm}
\bibliography{rotstarbiblio}

\end{document}